\documentclass[a4paper,10pt,reqno]{article}
\usepackage{amsthm}
\usepackage{amsmath,amssymb,amsfonts,amsxtra}
\usepackage{tikz}
\usepackage{pgflibraryarrows}
\usepackage{epic}
\usepackage{verbatim}
\usepackage[english]{babel}
\usepackage[T1]{fontenc}
\usepackage{lmodern}
\usepackage[mathscr]{euscript}
\usepackage{enumerate}
\usepackage{xcolor}
\definecolor{rouge}{rgb}{0.7,0,0}
\definecolor{bleu}{rgb}{0,0,0.7}
\usepackage[colorlinks=true,linkcolor=bleu,urlcolor=violet,citecolor=rouge]{hyperref}
\usepackage{breakurl}
\usepackage{xspace}
\usepackage{multirow}
\usepackage{longtable}
\usepackage{makeidx}
\usepackage{./youngtab}

\usetikzlibrary{decorations.markings}

\usepackage[all]{xy}
\makeatletter
\long\def\nnfoottext#1{\insert\footins{\footnotesize
    \interlinepenalty\interfootnotelinepenalty
    \splittopskip\footnotesep
    \splitmaxdepth \dp\strutbox \floatingpenalty \@MM
    \hsize\columnwidth \@parboxrestore
   \edef\@thefnmark{}
   \edef\@currentlabel{}\@makefntext
    {\rule{\z@}{\footnotesep}\ignorespaces
      #1\strut}}}
\makeatother

\numberwithin{equation}{section}

\newtheorem{thm}{Theorem}[section]

\newtheorem{prop}[thm]{Proposition}
\newtheorem{lm}[thm]{Lemma}
\newtheorem{cor}[thm]{Corollary}

\theoremstyle{definition}
\newtheorem{defi}[thm]{Definition}

\newtheorem{ex}[thm]{Example}

\newtheorem{Rqs}[thm]{Remarks}

\theoremstyle{remark}

\newtheorem*{mercis}{Acknowledgments}

\theoremstyle{plain}

\def\ad{\operatorname {ad}}

\def\codim{\operatorname {codim}}

\def\Hom{\operatorname {Hom}}

\def\GL{\operatorname {GL}}

\DeclareMathAlphabet{\calptmx}{OMS}{ztmcm}{m}{n}

\newcommand{\K}{{\Bbbk}}
\newcommand{\Z}{\mathbb{Z}}

\newcommand{\C}{\mathbb{C}}

\newcommand{\g}{\mathfrak{g}}
\newcommand{\h}{\mathfrak{h}}

\newcommand{\kk}{\mathfrak{k}}

\newcommand{\pp}{\mathfrak{p}}
\newcommand{\qq}{\mathfrak{q}}

\newcommand{\lf}{\mathfrak{l}}
\newcommand{\mf}{\mathfrak{m}}

\newcommand{\cc}{\mathfrak{c}}

\newcommand{\gl}{\mathfrak{gl}}
\newcommand{\sld}{\mathfrak{sl}_2}

\newcommand{\sln}{\mathfrak{sl}}
\newcommand{\so}{\mathfrak{so}}

\newcommand{\Striplet}{$\mathfrak{sl}_{2}$-triple\xspace}

\newcommand{\Od}{\mathcal{O}}

\newcommand{\sS}{\calptmx{S}}

\newcommand{\Id}{\mathrm{Id}}

\newcommand{\cpl}{\cc_{\pp}(\lf)^{\bullet}}

\newcommand{\cplnb}{\cc_{\pp}(\lf)}

\newcommand{\func}[4]{\left\{\begin{array}{r c l}#1&\rightarrow & #2\\ #3&\mapsto & #4\end{array}\right.}

\newcommand{\upind}{\uparrow \uparrow}
\newcommand{\wupind}{\uparrow}

\def\preisomto{\vbox{\hbox to
               14pt{\hfill$\sim$\hfill}\nointerlineskip\vskip -0.2pt
               \hbox to 14pt{\rightarrowfill}}}

\def\prelongisomto{\vbox{\hbox to
                17pt{\hfill$\sim$\hfill}\nointerlineskip\vskip -0.2pt
                \hbox to 17pt{\rightarrowfill}}}

\newenvironment{myitemize}
{ \begin{itemize}
    \setlength{\itemsep}{0pt}
    \setlength{\parskip}{0pt}
    \setlength{\parsep}{0pt}     }
{ \end{itemize}                  }

\makeindex

\begin{document}

\title{Sheets in symmetric Lie algebras and slice induction}
\author{{\sc Micha\"el~Bulois}
\thanks{{\url{Michael.Bulois@univ-st-etienne.fr}}, Universit\'e de Lyon,
CNRS UMR 5208,
Universit\'e Jean Monnet,
Institut Camille Jordan,
Maison de l'Universit\'e, 10 rue Tr\'efilerie,
CS 82301,
42023 Saint-Etienne Cedex 2,
France} \phantom{aaa}{\sc Pascal Hivert}\thanks{\url{pascal.hivert@math.uvsq.fr}, Universit\'e de Versailles-Saint-Quentin}}
\date{}
\maketitle

\nnfoottext{\!\!\!\!\!AMS Classification 2010: 17B70, 17B25

}

\begin{abstract}
In this paper, we study sheets of symmetric Lie algebras through their Slodowy slices. In particular, we introduce a notion of \emph{slice induction} of nilpotent orbits which coincides with the parabolic induction in the Lie algebra case. We also study in more details the sheets of the non-trivial symmetric Lie algebra of type G$_2$.  We characterize their singular loci and provide a nice desingularization lying in $\mathfrak so_7$.
\end{abstract}

\section{Introduction}
Let $\g$ be a reductive Lie algebra defined over an algebraically closed field of characteristic zero. Assume that $\g$ is $\Z/2\Z$-graded \[\g=\kk\oplus\pp,\] with even part $\kk$ and odd part $\pp$.
We may refer to such a symmetric Lie algebra as the \emph{symmetric pair} $(\g,\kk)$.

The algebraic adjoint group $G$ of $\g$ acts on $\g$ and the closed connected subgroup $K\subset G$ with Lie algebra $\kk$ acts on $\pp$.
A sheet of $\pp$ (resp. $\g$) is an irreducible component of a locally closed set of the form 
\[\pp^{(m)}:=\{x\in\pp\mid \dim K.x=m\}, \qquad (\mbox{resp. }\g^{(m)}:=\{x\in\g\mid \dim G.x=m\}) .\]

\medskip

Sheets of $\g$ have been extensively studied in several papers in the past decades. 

On one hand, the key papers of Borho and Kraft \cite{BK,Bo} describe sheets as disjoint union of so-called \emph{decomposition classes} (also known as \emph{Jordan classes}). The decomposition classes form a finite partition of $\g$, each class being  irreducible, locally closed and of constant orbit dimension. In particular, there is a dense class in each given sheet.
An important notion used in \cite{BK,Bo} is the parabolic induction of orbits introduced in \cite{LS} which gives rise to a notion of induction of decomposition classes. It is shown in \cite{Bo} that a sheet $S$ with dense decomposition class $J$ is precisely the union of the decomposition classes induced from $J$. Two consequences of this are the following. Firstly, each sheet contains a unique nilpotent orbit. Secondly, there is a parameterization of sheets coming with induction \cite[\S5]{BK}. Later, Broer developed methods for 
a deeper study of decomposition classes, their closures, regular closures, quotients and normalizations in \cite{Br1, Br2}.

On the other hand, making use of the previous parameterization, it is shown in \cite{Kat} that sheets are parameterized by their Slodowy slices. More explicitly, let $e$ be a representative of the nilpotent orbit of a sheet $S$ of $\g$ and embed $e$ in an $\sld$-triple $\sS=(e,h,f)$. The Slodowy slice of $S$ (with respect to $e$) is 
$$e+X(S,\sS):=S\cap (e+\g^f),$$
where $\g^f$ stands for the centraliser of $f$ in $\g$.
Katsylo proves that $S=G.(e+X(S,\sS))$ and that a geometric quotient of $S$ can be expressed as a finite quotient of $e+X(S,\sS)$.
In \cite{IH}, Im Hof shows that the morphism $G\times (e+X(S,\sS))\rightarrow S$ is smooth. This relates smoothness of $S$ to smoothness of $e+X(S,\sS)$ and eventually leads to a proof of smoothness of sheets in classical Lie algebras.

To our knowledge, the only known case of a singular sheet lies in a simple Lie algebra of type G$_2$ \cite{Pe, BK}. In this case, the two non-trivial sheets (\emph{i.e.} non-regular and with more than one orbit) are the two irreducible components of the set of subregular elements. One of these subregular sheets $S_2^{\g}$ is smooth while the other $S_1^{\g}$ is singular. More precisely, we can see that three analytical germs of $S_1^{\g}$ intersect in the neighborhood of elements of the subregular nilpotent orbit. In \cite[\S2]{Hi}, an explicit desingularization of $S_1^{\g}$ is constructed in terms of the well-known projection $\so_7\twoheadrightarrow \g$. 
\medskip

We now look at the symmetric case and sheets of $\pp$. Most of the ground results the authors are aware of in this setting are gathered in \cite[\S39.5-6]{TY}. An important feature is that there exists a notion of \emph{$K$-decomposition class} in $\pp$ and these classes share several good properties with the usual decomposition classes of $\g$. In particular, there is still a unique dense $K$-decomposition class in each sheet of $\pp$. In addition, each sheet of $\pp$ contains at least one nilpotent orbit. However the uniqueness statement no longer holds in general.  

One of the obstacles rising in the study of the sheets of $\pp$ is the lack of a well behaved notion of parabolic induction as one can check in the $(\sld,\so_2)$-case. For instance the induction theory developed in \cite{Oh} does not preserve orbit codimension, hence is of little help for our purpose. The main philosophy of \cite{Bu2} consists in noting that, at least in the case $\g=\gl_n$, the Slodowy slices of a sheet $S$ of $\pp$ still seem to encode significant geometric information of $S$. One of the aim of the present work is to justify this assertion in a more systematic manner. For instance, we show in Section \ref{slosli} that several properties of sheets, such as dimension, smoothness and orbits involved, are fully reflected in the corresponding Slodowy slices. We state these results for wider classes of subvarieties of $\pp$ in Proposition \ref{propsmooth}, Theorem \ref{genpart} and Proposition \ref{deformation1}.

Then we introduce in Section \ref{sliceind} the notion of slice induction. It turns out to be precise enough to rebuilt important parts of the theory resulting from parabolic induction in the Lie algebra case. This includes (see Theorem \ref{thmind} and Corollary \ref{corind}):
\begin{myitemize}
\item Construction of one parameter deformations of orbits.
\item Stratification properties for decomposition classes.
\item Characterization of a sheet with dense $K$-decomposition class $J$ as the union of induced classes from $J$.
\end{myitemize}
In the Lie algebra case, this also provides a new insight on these results.

In \cite{Bu2}, the connection between sheets of $\g$ and sheets of $\pp$ was studied.  In particular it was shown that, whenever sheets of $\g$ are smooth, sheets of $\pp$ are also smooth. The last goal of this paper is the description of sheets of $\pp$ in the non-trivial symmetric Lie algebra of type G$_2$. We make use of two approaches for this study. In Section \ref{G2slsl}, we study the subregular sheets through their Slodowy slices. This provides the set-theoretical description of the sheets and describes the behavior of the singularities of $S_1^{\g}$ through the intersection with $\pp$. In Section \ref{G24ality}, we exploit the symmetry of $\g$ making use of $4$-ality as described in \cite{LM}. This allows us to construct a nice desingularization in $\so_7$ of the singular sheet (Proposition \ref{desing}), following the guidelines of \cite{Hi}. 

\medskip

It is plausible that most of what is stated in Section \ref{slosli} and \ref{sliceind} remains true in the more general setting of $\theta$-representation. However, the authors are unaware of references for a general theory of decomposition classes in this setting.

\begin{mercis}
We are grateful to the referees for the fast and careful reading of the paper. In particular, Section \ref{sliceind} has  benefited from their remarks.
\end{mercis}

\section{Geometry of subvarieties and Slodowy slices}
 \label{slosli}
We start with some notation. In the whole paper, $\K$ is an algebraically closed field of characteristic $0$, $\g=\kk\oplus\pp$ is a reductive symmetric Lie algebra over $\K$, that is a $\Z/2\Z$-graded reductive Lie algebra $\g$ with even part $\kk$ and odd part $\pp$. In particular, $\kk$ is a Lie algebra and $\pp$ is a $\kk$-module. We denote by $\g'$ the semisimple part of $\g$. Lie algebras can be seen as particular cases of symmetric Lie algebras in the following sense: given a Lie algebra $\g$, there exists a symmetric Lie algebra $\hat\g=\hat\kk\oplus\hat\pp$ such that the $\hat\kk$-module $\hat\pp$ is isomorphic to the $\g$-module $\g$\footnote{Namely $\hat\g= \g\times \g$, $\hat \kk=\{(x,x)|x\in \g\}$ and $\hat \pp=\{(x,-x)|x\in \g\}$.}. As a consequence, all the statements enounced below in the symmetric setting hold for Lie algebras replacing both $\kk$ and $\pp$ by $\g$, and $K$ by $G$.
A large part of the Lie theory have a symmetric counterpart. We refer to \cite{KR} for the ground results on symmetric Lie algebra. 

Let $G$ be the adjoint group of $\g$ and $K$ be the closed connected subgroup of $G$ with Lie algebra $\kk\cap \g'$. The group $K$ acts on $\pp$.
For $x\in \pp$, it follows from \cite[Proposition~5]{KR} that \begin{equation}\dim K.x=\frac 12 \dim G.x, \qquad \dim \kk-\dim \kk^x=\dim \pp-\dim \pp^x.\label{egalfond} \end{equation}
For $A\subset\g$, we set $A^{\bullet}:=\{a\in A| \,\forall a'\in A, \, \dim G.a\geqslant \dim G.a'\}$. Note that we can replace $G$ by $K$ in the previous definition when $A\subset \pp$ thanks to \eqref{egalfond}. 
For $A,B\subset\g$, the centralizer in $B$ of $A$ is denoted by $\cc_{B}(A):=\{b\in B|\, \forall a\in A, \,  [a,b]=0\}$.

If $\mf$ is a $\Z/2\Z$-subspace of $\g$ we write $\kk_{\mf}:=\kk\cap \mf$, $\pp_{\mf}:=\pp\cap\mf$ and we have \[\mf=\kk_{\mf}\oplus \pp_{\mf}.\]
We say that a Levi subalgebra $\lf\subset\g$ \emph{arises from $\pp$} if there exists a semisimple element $v\in \pp$ such that $\lf=\g^v$ (it corresponds to the notion of subsymmetric pair in \cite{PY}). In this case, $\lf$ and $\lf'$ are reductive and semisimple $\Z/2\Z$-graded Lie subalgebras of $\g$. In addition, we can decompose $\g$ in $\lf$-modules in the following way
\[\g=\lf\oplus\lf^{\perp}\]
where $\lf^{\perp}$ is the orthogonal complement of $\lf$ in $\g'$ with respect to the Killing form. More concretely, if $\lf=\g^v$ we can write $\lf^{\perp}=[\g,v]$.

It is well-known that $\cc_{\g}(\lf)$ is the center of $\lf$.
In particular, $\pp_{\lf}=\cplnb\oplus\pp_{\lf'}$. 
Moreover, we have \cite[38.8.4]{TY}
\begin{eqnarray}&\cc_{\pp}(\lf)=\cc_{\g}(\lf)\cap\pp=\cc_{\pp}(\pp_{\lf}),&\label{cplnobullet}\\
&\cc_{\pp}(\lf)^{\bullet}=\cc_{\g}(\lf)^{\bullet}\cap\pp=\{u\in \cc_{\pp}(\lf)|\, \g^u=\lf\}. &\label{cplbullet}\end{eqnarray}
We denote by $K_{\lf}$ (resp. $K_{\lf'}$) the closed connected subgroup of $K$ with Lie algebra $\kk_{\lf}\cap \g'$ (resp. $\kk_{\lf'}$). Then, $K_{\lf}=(K^v)^{\circ}$ and the $K_{\lf}$-orbits of $\pp_{\lf}$ are precisely the orbits associated to the reductive symmetric Lie algebra $\lf=\kk_{\lf}\oplus\pp_{\lf}$. The same holds for $K_{\lf'}$-orbits (=$K_{\lf}$-orbits) of $\pp_{\lf'}$.
Define \begin{equation}\label{defUl}U_{\lf}:=\{y\in \pp_{\lf}\,|\,\g^y\subset \lf\}.\end{equation}
The next lemma shows that $K$-orbits and $K_{\lf}$-orbits of elements of $U_{\lf}$ are closely related.

\begin{lm}\label{ts}
Let $\lf$ be a Levi subalgebra arising from $\pp$. Then, the following conditions are equivalent for any $y\in \pp_{\lf}$.
\begin{enumerate}[(i)]
\item $y\in U_{\lf}$,
\item $\g^s\subset \lf$, where $s$ is the semisimple component of $y$,
\item $\codim_{\pp}K.y=\codim_{\pp_{\lf}}K_{\lf}.y$,
\item $[\kk,y]=\pp_{\lf^{\perp}}\oplus[\kk_{\lf},y]$.
\end{enumerate}
Moreover, $U_{\lf}$ is an open subset of $\pp_{\lf}$.
\end{lm}

\begin{proof}
We have $y\in \lf$ . Hence $[\g,y]=[\lf^{\perp},y]\oplus[\lf,y]\subset \lf^{\perp}\oplus [\lf,y]$, with equality if and only if $\g^{y}\cap \lf^{\perp}=\{0\}$ if and only if $\g^{y}\subset \lf$. As a consequence, (i) is equivalent to (iv'): $[\g,y]=\lf^{\perp}\oplus [\lf,y]$. By the way, we also note that $\g^{y}\cap \lf^{\perp}=\{0\}$ is an open condition on $y$.

On the other hand, we see that (iv) is just (iv') intersected with $\pp$. Let $v\in \pp$ such that $\lf=\g^v$. With the help of \eqref{egalfond}, we have \begin{eqnarray*}\dim (\pp_{\lf^{\perp}}\oplus[\kk_{\lf},y])-\dim [\kk,y]&=&\dim[\kk,v]+\dim [\kk_{\lf},y]-\dim [\kk,y]\\
&=&\frac 12(\dim[\g,v]+\dim [\lf,y]-\dim [\g,y])\\
&=& \frac12(\dim (\lf^{\perp}\oplus[\lf,y])-\dim [\g,y])\end{eqnarray*} Since we have the inclusion $[\kk,y]\subset\pp_{\lf^{\perp}}\oplus[\kk_{\lf},y]$, we get the equivalence between (iv') and (iv). Through tangeant spaces, we also see that (iii) is equivalent to (iv).
%
%
%
%
%

There remains to show that (i) is equivalent to (ii). Denote the nilpotent part of $y$ by $n$. We have $\g^y= \g^s\cap \g^n$ so (ii) implies (i). Let us now assume that $\g^s$ is not included in $\lf$. Since $s,n\in \lf$, the endomorphism $\ad n$ stabilizes the non-trivial subspace $\g^s\cap \lf^{\perp}$. Since $\ad n$ is nilpotent, there exists a non-zero element in $\g^n\cap\g^s\cap \lf^{\perp}$. Hence, $\g^{y}\not \subset \lf$. By contraposition, (i) implies (ii) .
%
%
%
%
%
%
\end{proof}


\begin{defi}\label{defXl}Given a Levi subalgebra $\lf$ of $\g$ arising from $\pp$, a nilpotent element $e\in\pp_{\lf}$ embedded in a normal $\sld$-triple $\sS:=(e,h,f)\subset \lf$ (here \emph{normal} means $e,f\in\pp$, $h\in\kk$ and we allow $(0,0,0)$ as $\sld$-triple) and a subset $J\subset\pp$, we define $X_{\lf}(J,\sS)\subset \pp_{\lf}^f$ via 
\[e+X_{\lf}(J,\sS):=(e+\pp_{\lf}^f)\cap U_{\lf}\cap J,\]
where $U_{\lf}$ is as in \eqref{defUl}. 
We say that $e+X_{\lf}(J,\sS)$  is the \emph{Slodowy slice} of $J$ with respect to $(\lf,\sS)$.
\end{defi}
This is a generalization of the notion appearing in \cite{Sl, Kat, Bu2}. In fact, in the case $\lf=\g$, we have $U_{\lf}=\g$ and \[e+X(J,\sS):=e+X_{\g}(J,\sS)=(e+\pp^f)\cap J\] is the natural analogue in the symmetric setting of the ordinary Slodowy slice in Lie algebras, see \cite{Bu2}. 

In what follows, a cone of $\pp$ means a subset of $\pp$ stable under multiplication by $\K^*$.
The next lemma is well-known (see \emph{e.g.} \cite[\S4]{Kat}). 
\begin{lm}\label{lmelimX} Let $J$ be a $K$-stable cone of $\pp$ and $\sS=(e,h,f)$ a normal $\sld$-triple such that $X(J,\sS)\neq\emptyset$.\\ Then $e$ belongs to each irreducible component of $e+\overline{X(J,\sS)}$.
In particular, $e\in \overline{J}$.
\end{lm}
\begin{proof}
A standard proof of this lemma is based on the construction of a one-parameter subgroup of $K\times \K^*\Id$ 
which contracts $e+X(J,\sS)$ to $e$. Define the characteristic grading $\g:=\bigoplus_{i\in\Z} \g(i,h)$ by $\g(i,h):=\{x\in \g\mid [h,x]=ix\}$. For ${t\in\K^*}$, let $F_t\in \GL(\g)$ be such that $$(F_t)_{\mid\g(i,h)}=t^{-i+2} \Id.$$ 
We have $F_t.e=e$ and $F_t.(e+\pp^f)=e+\pp^f$ since $\pp^f$ is compatible with the characteristic grading. 
On the other hand, it is easy to show \cite[38.6.2]{TY} that $t^{-2}F_t\in K$, hence $F_t\in K\times\K^*\Id$. As a consequence, $F_t$ normalizes the cone $J$ and hence $(e+X(J,\sS))$. Since $e+\pp^f\subset e+\bigoplus_{i\leqslant 0}\g(i,h)$, we have 
\begin{equation}\label{Ft}\lim_{t\rightarrow 0} F_t.(e+x)=e\end{equation} for any $x\in \pp^f$. Since $(F_t)_{t\in\K^*}$ is a one parameter subgroup of $\GL(\g)$, each irreducible component of $e+X(J,\sS)$ is stable under the $F_t$-action and the lemma follows.\end{proof}

Next, we wish to enlight the strong connection linking $J$ and $X(J,\sS)$. This was lacking in \cite{Bu2} in the general case and the following Proposition renders some definitions and techniques of this paper obsolete. For example, it can easily be used together with \cite[38.6.9(i)]{TY} to show that condition ($\clubsuit$) of \cite[\S9]{Bu2} is automatically satisfied.

\begin{prop}\label{propsmooth} Let $\lf$ be a Levi subalgebra of $\g$ arising from $\pp$ and $\sS=(e,h,f)$ a normal $\sld$-triple of $\lf'$. Let $J$ be an irreducible locally closed $K$-stable subset of $\pp$ such that $X_{\lf}(J,\sS)\neq\emptyset$ and $Y$ be a locally closed subset in $X_{\lf}(\pp,\sS)$. Set $c(J):=\codim_\pp J$, $c(Y):=\codim_{X_{\lf}(\pp,\sS)}Y$ and $r:=\codim_{\pp}X_{\lf}(\pp,\sS)$.
\begin{enumerate}
\item[(i)] The action of $K$ on $\pp$ yields a smooth morphism $\psi:K\times(e+X_{\lf}(\pp,\sS))\rightarrow \pp$ of relative dimension $\dim K-r$.
\item[(ii)] $\codim_{\pp}K.(e+Y)\leqslant c(Y)$.
\item[(iii)] $X_{\lf}(J,\sS)$ is a pure locally closed variety of codimension $c(J)$ in $X_{\lf}(\pp,\sS)$.  
\item[(iv)] Let $X_0$ be an irreducible component of $X_{\lf}(J,\sS)$. Then $K.(e+X_0)$ is dense in $J$. 
\item[(v)] $\psi$ restricts to a smooth dominant morphism \[\psi_J:K\times(e+X_{\lf}(J,\sS))\rightarrow J.\]
\end{enumerate}
\end{prop}
\begin{proof}
First of all, $\psi$ is $K$-equivariant with respect to the $K$-action on the domain of $\psi$ given by $k'.(k,y)=(k'k,y)$. Hence it is sufficient to check that $\psi$ is smooth at points of the form $(1_K,y)$. Moreover, since the domain and codomain of $\psi$ are smooth (recall from Lemma \ref{ts} that $e+X_{\lf}(\pp,\sS)$ is open in the affine space $e+\pp_{\lf}^f$), it is sufficient to check that the induced map on the tangeant spaces is surjective \cite[VII Remark 1.2]{AK}.
At the point $(1_K,y)$, it is given by $(d_{\psi})_{(1,y)}:\func{\kk\times \pp_{\lf}^f}{\pp}{(k,x)}{[k,y]+x}$.
Hence, thanks to Lemma \ref{ts}, we have $(d_{\psi})_{(1,y)}(\kk\times \pp_{\lf}^f))=[\kk,y]\oplus \pp_{\lf}^f=\pp_{\lf^{\perp}}\oplus [\kk_{\lf},y]\oplus\pp_{\lf}^f$. 
This implies that $d_\psi$  is smooth at $(1,y)$ if and only if $d_{\psi'}$ is, where $\psi':K_{\lf}\times(e+\pp_{\lf}^f) \rightarrow \pp_{\lf}$ is the morphism yielded by the action of $K_{\lf}$ on $\pp_{\lf}$. In other words, we can restrict ourselves to the case $\g=\lf$.

In this case, we follow \cite[7.4 Corollary~1]{Sl}. We easily see from graded $\sld$-theory that $\pp=[\kk,e]\oplus\pp^f$ so $\psi$ is smooth at $(1,e)$. Consider the $(F_{t})_{t\in\K^*}$-action on $K\times (e+\pp^f)$, given by $F_t.(k,x)=(F_tkF_{t^{-1}},F_{t}(x))$. Then $\psi$ is equivariant with respect to this action. Hence the open set of smooth points of $\psi$ is stable under $F_t$. 
Therefore, it follows from \eqref{Ft} that $\psi$ is smooth on $1_{K}\times(e+\pp^f)$ and hence on the whole domain $K\times(e+\pp^f)$. \\
(ii) We have $K.(e+Y)=\psi(K\times(e+Y))$. Assertion (i) implies that $\psi$ has constant fiber dimension. Hence, the dimension of any fiber of $\psi_{|K\times(e+Y)}$ is less or equal to $\dim K-r$ and we have $\dim K.(e+Y)\geqslant\dim K+\dim Y-(\dim K-r)=\dim Y+r
=\dim \pp-c(Y)$.\\
(iii-iv) \cite[I Proposition 7.1]{Ha} (which also holds for locally closed subsets of an affine space)  states that $\codim_{\pp} X_0\leqslant c(J)+r$ for any irreducible component $X_0$ of $X_{\lf}(J,\sS)$. On the other hand, since $J$ is $K$-stable, we have $K.(e+X_0)\subset J$ and we deduce $\codim_{X_{\lf}(\pp,\sS)}X_0\geqslant c(J)$ from (ii). Since $J$ is irreducible, this proves (iii) and (iv).\\
(v) It follows from (iv) that $\psi_J$ is dominant. In order to obtain smoothness, we apply the argument of base extension by $J\hookrightarrow \pp$ as in \cite[2.8 and above]{IH}. More details can also be found in \cite[Proposition 3.9]{Bu2}. 
\end{proof}

\begin{prop}\label{propopen}
Let $\lf,J,\sS$ be as in Proposition \ref{propsmooth}, omitting the assumption $X_{\lf}(J,\sS)\neq\emptyset$. Then
\begin{enumerate}
\item[(i)] $X_{\lf}(J,\sS)$ is a dense open subset of $X_{\lf}(\overline{J},\sS)$,
\item[(ii)] $K.(e+X_{\lf}(J,\sS))$ is an open subset of $J$. 
\end{enumerate}
\end{prop}
\begin{proof}
(i) If $X_{\lf}(\overline{J},\sS)=\emptyset$, there is nothing to prove. From now on, we assume that it is non-empty. Since $J$ is open in $\overline{J}$, the subset $X_{\lf}(J,\sS)$ is open in $X_{\lf}(\overline{J},\sS)$.
Let $X_0$ be an irreducible component of $X_{\lf}(\overline{J},\sS)$. Then $K.(e+X_0)$ is a dense constructible subset of $\overline{J}$ by Propostion \ref{propsmooth}(iv). Hence it meets $J$ and, since $J$ is $K$-stable, we have $J\cap X_0\neq\emptyset$. In other words, the open subset $X_{\lf}(J,\sS)\subset X_{\lf}(\overline{J},\sS)$ meets each irreducible components of $X_{\lf}(\overline{J},\sS)$.\\
(ii) Since smooth morphisms are open \cite[V Theorem 5.1 and VII Theorem 1.8]{AK}, the result is a consequence of (i) and Proposition \ref{propsmooth}(v).
\end{proof}

The equivalent statements (i), (ii) and (iii) of the following Theorem are inspired by similar properties in the Lie algebra case, which can be deduced from parabolic induction theory when $J$ is a decomposition class (see \emph{e.g.} \cite[\S2]{Bo}).

\begin{thm}\label{genpart}
Let $e\in \pp$ be a nilpotent element embedded in a normal $\sld$-triple $\sS:=(e,h,f)$ and let $J$ be a locally closed $K$-stable cone. The following conditions are equivalent:
\begin{enumerate}
\item[(i)] $e\in \overline{J}$.
\item[(ii)] There exists a non-empty open set $U\subset J$ such that $e\in \overline{K.(\K^*z)}$ for any $z\in U$.
\item[(iii)] There exists $z\in J$ such that $e\in \overline{K.(\K^*z)}$.
\item[(iv)] $X(J,\sS)\neq\emptyset$.
\item[(v)] $X(\overline{J},\sS)\neq\emptyset$.
\end{enumerate}
\end{thm}
\begin{proof}
The implications (ii)$\Rightarrow$(iii)$\Rightarrow$(i)$\Rightarrow$(v) are obvious and (v)$\Rightarrow$(iv) is a consequence of Proposition \ref{propopen}(i) applied to an irreducible component of $\overline{J}$ meeting $e+\pp^f$.\\
Let us now prove (iv)$\Rightarrow$(ii). Take $U=K.(e+X(J,\sS))$. It is open in J by Proposition \ref{propopen}(ii) and, for any $z\in U$, we have $X(K.(\K^*z), \sS)\neq\emptyset$. Then, our implication is a consequence of Lemma \ref{lmelimX} applied to the $K$-stable cone $K.(\K^*z)$.
\end{proof}

From the computational point of view, the previous theorem may be of key importance. Indeed, the existence of a degeneration from $J$ to $e$ reduces to the existence of an element in the intersection of $J$ with the affine space $e+\pp^f$, which might be much easier to check.

On the other hand, we have the following proposition derived from \cite[Theorem 38.6.9 (i)]{TY}.

\begin{prop}\label{deformation0}
If $J$ is a $K$-stable cone of $\pp^{(r)}=\{x\in \pp|\dim K.x=r\}$, then there exists a nilpotent element $e\subset \pp$ such that 
\begin{enumerate}
\item[(i)] $e\in \overline{J}$,
\item[(ii)] $\dim K.e=r$.
\end{enumerate}
\end{prop}

\begin{prop}\label{deformation1}
Let $J\subset \pp^{(r)}$ be a locally closed $K$-stable cone and $(e_i)_{i\in I}$ a set of representatives of the nilpotent $K$-orbits satisfying (i) and (ii) of Proposition \ref{deformation0} embedded in normal $\sld$-triples $(\sS_i)_{i\in I}$. Let $U_i:=K.(e_i+X(J,\sS_i))$.\\
 Then, $(U_i)_{i\in I}$ is an open cover of $J$.
\end{prop}
\begin{proof}
The sets $U_i$ are open thanks to Proposition \ref{propopen}(ii). Pick $z\in J$, then it follows from Proposition \ref{deformation0} that there exists $i\in I$ such that $e_i\in \overline{K.(\K^*z)}$. Applying Theorem~\ref{genpart} to $K.(\K^*z)$, we get $X(K.(\K^*z), \sS_i)\neq\emptyset$, so $U_i\cap K.(\K^*z)\neq\emptyset$. Since $U_i$ is stable under the action of $K\times \K^*\Id$, we get $z\in U_i$.
\end{proof}

\begin{Rqs}\label{smoothequiv}
\begin{enumerate}[-]
\item A major consequence of this Proposition is that the whole geometry of $J$ is closely related to the geometry of its different Slodowy slices. Indeed, locally, we can assume that $J=K.(e_i+X(J,\sS_i))$ for some $i\in I$ and it follows from Proposition \ref{propsmooth}(v) that this variety is smoothly equivalent to $X(J,\sS_i)$. 
\item This also provides a more solid ground to the philosophy of \cite[\S9]{Bu2}, where it is proven in some particular cases that the Slodowy slices contains enough information to describe the whole variety.  
\item The main drawback of this approach is that it does not yield $|I|=1$ when $J$ is irreducible in the Lie algebra case, contrary to other parameterization such as \cite[\S5]{BK}. However, this drawback is somehow necessary since the property $|I|=1$ may fail in the symmetric case (see \emph{e.g.} \cite[39.6.3]{TY} for the description of the regular sheet when $(\g,\kk)=(\sld,\so_2)$). 
\end{enumerate}
\end{Rqs}

We have several examples in mind of such locally closed $K$-stable cone included in some $\pp^{(r)}$. We have already seen that $K.(\K^*z)$ plays a role in the previous proofs. We can also consider a $K$-decomposition class $J_1$, or its regular closure $\overline{J_1}^{\bullet}$. Sheets are particular examples of the last type.

\section{$K$-decomposition classes and induction}
\label{sliceind}
Recall from the beginning of Section~\ref{slosli} that a Levi subalgebra $\lf$ is said to \emph{arise from $\pp$} if it is the centralizer of a semisimple element $s\in \pp$. Recall also that, in this case, the orbits in $\pp_{\lf}$ associated to the symmetric Lie algebra structure $\lf=\kk_{\lf}\oplus\pp_{\lf}$ coincide with the $K_{\lf}=(K^s)^{\circ}$-orbits.
\begin{defi}
A \emph{datum} of $(\g,\kk)$ is a pair $(\lf,\Od)$ where $\lf$ is a Levi subalgebra of $\g$ arising from $\pp$ and $\Od$ is a nilpotent $K_{\lf}$-orbit in $\pp_{\lf}$.

The \emph{$K$-decomposition class} defined by a datum $(\lf,\Od)$ is
\[J(\lf,\Od):=J_{(\g,\kk)}(\lf,\Od):=K.(\cpl+\Od)
.\]
\end{defi}
\begin{Rqs}\label{rqdec}
\begin{enumerate}[-]
\item These definitions make sense with $(\g,\kk)$ replaced by any other reductive other reductive symmetric pair. For instance, if $\lf_2\subset \g$ is a Levi subalgebra arising from $\pp$ and $(\lf_1,\Od_1)$ is a datum of $(\lf_2, \kk_{\lf_2})$, we have \[ J_{(\lf_2,\kk_{\lf_2})}(\lf_1,\Od_1)= K_{\lf_2}.(\cc_{\pp_{\lf_2}}(\lf_1)^{reg}+\Od_1)=K_{\lf_2}.(\cc_{\pp}(\lf_1)^{reg}+\Od_1)
\]
where $A^{reg}:=\{a\in A| \forall a'\in A,  \dim K_{\lf_2}.a\geqslant \dim K_{\lf_2}.a'\}$ is a regularity condition taken with respect to the action of $K_{\lf_2}$.
\item 
The classes we have defined are the symmetric analogue of decomposition classes studied in \cite{BK, Bo, Br1, Br2} in the Lie algebra case. We call these last classes \emph{$G$-decomposition classes} and they are of the form
$J_{\g}(\lf,\Od):=G.(\cc_{\g}(\lf)^{\bullet}+\Od)$ with $\lf\subseteq\g$ any Levi subalgebra and $\Od$ a nilpotent $L$-orbit in $\lf$. In fact, using notation of Section \ref{slosli}, the $G$-decomposition classes are exactly the images of the $\widehat K$-decomposition classes under the natural isomorphism $\hat \pp\cong\g$. These are of the form $J_{(\hat \g,\hat \kk)}(\hat\lf,\hat\Od)$ with $\hat \lf=\{(x,y)| x\in \lf, y\in \lf\}$ and $\hat \Od=\{(x,-x)| x\in \Od\}$. The connection between a $K$-decomposition class $J$ and the $G$-decomposition class containing $J$ is studied in \cite[\S8]{Bu2}.
%
%
\item The $K$-orbits of Levi-factors arising from $\pp$ are in one to one correspondance with their $G$-orbits which are, in turn, easily characterized through the Satake diagram of $(\g,\kk)$, cf. \cite[\S7]{Bu2}. Hence, in order to have a convenient classification of $K$-decomposition classes, it would be enough to understand in which cases $(\lf,\mathcal O_1)$ is $N_{K}(\lf)/K_{\lf}$-conjugate to $(\lf, \mathcal O_2)$. This plays an important role in the classification of sheets when $\mathcal O_i$, $i=1,2$ are rigid orbits as shown in \cite[3.9, 3.10, 4.5, 4.6]{Bo}. However, for our purposes, we will not need such a classification.
\end{enumerate}
\end{Rqs}

The $K$-decomposition classes are equivalence classes of elements sharing a similar Jordan decomposition in the following sense: 
\begin{equation}J(\lf,\Od)=K.\{x=s+n\in \pp\,|\, \g^s=\lf, K_{\lf}.n=\Od\}\label{caracJc}\end{equation} 
and we refer to \cite[\S39.5]{TY} for most of the known properties of these classes (called $K$-Jordan classes in \emph{loc. cit.}) in the symmetric Lie algebra case. 
In particular, we note that $K$-decomposition classes form a finite partition of $\pp$. These classes are also irreducible, locally closed and of constant orbit dimension. As a consequence, in each sheet $S$ there exists a dense open $K$-decomposition class $J_0$ and we have \[S=\overline{J_0}^{\bullet}.\]

In the Lie algebra case, such regular closures of $G$-decomposition classes have been studied in \cite{BK}, where an important parameterization is given in Theorem~5.4 (see \cite[Lemma~3.2]{Kat} for a refinement). In \cite{Bo}, this understanding has been deepened thanks to the parabolic induction theory of nilpotent orbits. Let us recall the definition. Take two Levi subalgebras $\lf_1 \subseteq \lf_2\subseteq \g$, with corresponding subgroups $L_i\subseteq G$, $i\in\{1,2\}$. One can choose a parabolic subalgebra $\qq$ of $\lf_2$ with Levi-factor $\lf_1$. Denote by  $\mathfrak{n}_{\qq}$ the nilpotent radical of $\qq$.  Given a nilpotent $L_1$-orbit $\Od_1\subset \lf_1$, \emph{the} parabolically induced orbit from $\lf_1$ to $\lf_2$ is the dense nilpotent $L_2$-orbit $\Od_2$ of $L_2.(\Od_1+\mathfrak{n}_{\qq})\subset\lf_2$. For the purpose of our exposition, we will say in such case that the pair $(\lf_1, \Od_1)$ (resp. the $G$-decomposition class $J_{\g}(\lf_1,\Od_1)$) \emph{parabolically induces} $(\lf_2,\Od_2)$ (resp. $J_{\g}(\lf_2,\Od_2)$).

The aim of this section is to extend a significant number of the known properties of $G$-decomposition classes of \cite{BK, Bo} to $K$-decomposition classes. However, as mentioned in the introduction, the parabolic subalgebras behave badly in general for a symmetric Lie algebra. 
This is why we introduce a new notion of induction, using the tools developed in Section \ref{slosli}. This notion turns out to coincide with the parabolic induction in the Lie algebra case, see Corollary~\ref{corind}.


 


\begin{defi}\label{definduction}
\begin{enumerate}
\item \label{weakind} Given two data $(\lf_i,\Od_i)$ ($i=1,2$) of $(\g,\kk)$, we say that $(\lf_1,\Od_1)$ \emph{weakly slice induces} $(\lf_2,\Od_2)$ if $\lf_1\subseteq\lf_2$ and \begin{equation}J_{(\lf_2,\kk_{\lf_2})}(\lf_1,\Od_1)\cap (n_2+\pp_{\lf_2}^{m_2})\neq\emptyset\label{condind}\end{equation} where  $n_2$ is an element of $\Od_2$ embedded in a normal $\sld$-triple $\sS_2=(n_2,h_2,m_2)$ of $\lf_2$.
\item  \label{mainind} We say that $(\lf_1,\Od_1)$ \emph{slice induces} $(\lf_2,\Od_2)$ if, moreover, $\dim K_{\lf_2}.x=\dim \Od_2$ for some $x\in J_{(\lf_2,\kk_{\lf_2})}(\lf_1,\Od_1)$.

\item \label{Jordind} Given two $K$-decomposition classes $J_i$ ($i=1,2$), we say that $J_1$ (weakly) slice induces $J_2$ if there exist data $(\lf_i,\Od_i)$ with $J_i=J(\lf_i,\Od_i)$, ($i=1,2$) such that  $(\lf_1,\Od_1)$ (weakly) slice induces $(\lf_2,\Od_2)$.
\end{enumerate}
\end{defi}

When, there is no context of parabolic induction, the term \emph{induction} will always refer to slice induction.

When $J_1$ and $J_2$ are $K$-decomposition classes, we write \[\left\{\begin{array} {l l} J_1\wupind J_2& \textrm{if $J_1$ weakly slice induces $J_2$,}\\ J_1\upind J_2 & \textrm{if $J_1$ slice induces $J_2$.} \end{array}\right.\] 

\begin{Rqs}\label{rkdefind}
\begin{enumerate}[-]
\item The definitions do not depend on the choices of $n_2$ or $\sS_2$ since all of these are $K_{\lf_2}$-conjugate. Also, all the elements of $J_{(\lf_2, \kk_{\lf_2})}(\lf_1,\Od_1)$ share the same $K_{\lf_2}$-orbit dimension. Finally, it is worth noting that the notions of induction and weak induction of data do only depend on the involved data and not on the ambient Lie algebra $\g$.
\item If $\lf_2=\g$, condition \ref{condind} can be written in the nicer form:
\begin{equation}\label{condindbis}X(J(\lf_1,\Od_1), \sS_2)\neq\emptyset.\end{equation}
\item Assume that $\lf_1\subseteq\lf_2$. The data $(\lf_1,\Od_1)$ (weakly) induces $(\lf_2, \Od_2)$ if and only if $(\lf_1\cap \lf_2',\Od_1)$ (weakly) induces $(\lf_2', \Od_2)$. Indeed, $n_2+\pp_{\lf_2}^{m_2}=\cc_{\pp}(\lf_2)+(n_2+\pp_{\lf_2'}^{m_2})$ and 
$J_{(\lf_2, \kk_{\lf_2})}(\lf_1,\Od_1)=\cc_{\pp}(\lf_2)+J_{(\lf_2', \kk_{\lf_2'})}(\lf_1\cap \lf_2', \Od_1)$. In particular, when studying induction, one can always assume that the ambient Lie algebra $\g$ is semisimple.
\end{enumerate}
\end{Rqs}

Assume that $\g$ is semisimple. In view of Theorem \ref{genpart} and \eqref{condindbis}, a $K$-decomposition class $J_1$ weakly induces a nilpotent orbit $\Od$ (\emph{i.e.} the $K$-decomposition class $J(\g,\Od)$) if and only if $\Od\subset \overline{J_1}$. 

The purpose of the following theorem is to extend this property to induction of arbitrary $K$-decomposition classes. Recall that $X_{\lf}(J,\sS)$ is introduced in Definition \ref{defXl}.

\begin{thm}\label{thmind} Let $J_1,J_2$ be $K$-decomposition classes. The following conditions are equivalent.
\begin{enumerate}[(i)] 
\item \label{wi} $J_1\wupind J_2$,
\item \label{wii} $X_{\lf_2}(J_1,\sS_2)\neq\emptyset$ where $(\lf_2, \Od_2)$ is some (any) datum defining $J_2$ and $\sS_2$ is a normal $\sld$-triple in $\lf_2$ whose nilpositive element belongs to $\Od_2$.
\item \label{wiii} $J_2\cap \overline{J_1}\neq\emptyset$,
\item \label{wiv} $J_2\subset \overline{J_1}$,
\end{enumerate}
\end{thm}
\begin{proof}
\eqref{wiv}$\Rightarrow$ \eqref{wiii} is obvious.

\eqref{wiii}$\Rightarrow$ \eqref{wii}: Choose $n_2\in \Od_2$ and embed it in a normal $\sld$-triple $\sS_2$ of $\lf_2$. Since $J_2\cap \overline{J_1}$ is $K$-stable we can choose an element $y_2$ in this intersection such that the Jordan decomposition of $y_2$ is $s_2+n_2$ with $s_2\in \cc_{\pp}(\lf_2)^{\bullet}$ and $n_2\in \Od_2$. Therefore $\g^{s_2}=\lf_2$ \eqref{cplbullet} so $y_2\in U_{\lf_2}$ and $X_{\lf_2}(\overline{J_1},\sS_2)\neq\emptyset$. Then Proposition~\ref{propopen}(i) yields $X_{\lf_2}(J_1,\sS_2)\neq\emptyset$.

\eqref{wii}$\Rightarrow$\eqref{wi}: Choose an element $y_1\in X_{\lf_2}(J_1, \sS_2)$ and consider its Jordan decomposition $s_1+n_1$. Since $y_1\in U_{\lf_2}$, the centralizer $\lf_1:=\g^{s_1}$ is a Levi subalgebra of $\lf_2$ (Lemma~\ref{ts}). Hence it makes sense to speak of the $K_{\lf_2}$-decomposition class $J_{(\lf_2,\kk_{\lf_2})}(\lf_1,\Od_1)$ with $\Od_1:=(K^{s_1})^{\circ}.n_1$. Since $\lf_1=\lf_2^{s_1}$, we derive from \eqref{caracJc} that $y_1$ belongs to this class. Since $y_1\in n_2+\pp_{\lf_2}^{m_2}$, the datum $(\lf_1,\Od_1)$ weakly induces $(\lf_2,\Od_2)$.

\eqref{wi}$\Rightarrow$\eqref{wiv}: Choose a datum $(\lf_1,K_{\lf_1}.n_1)$ defining $J_1$, which weakly induces a datum $(\lf_2,K_{\lf_2}.n_2)$ defining $J_2$. Since $\lf_1\subset\lf_2$, we have $\cc_{\pp}(\lf_1)\supset\cc_{\pp}(\lf_2)$ and 
\[\overline{J_1}=\overline{K.(\cc_{\pp}(\lf_1)^{\bullet}+n_1)}\supset \overline{K_{\lf_2}.(\cc_{\pp}(\lf_2)+\cc_{\pp_{\lf_2'}}(\lf_1)+n_1)}=\cc_{\pp}(\lf_2)+\overline{J_1'}\]
where $J_1'$ is the $K_{\lf_2'}$-decomposition class $J_{(\lf_2', \kk_{\lf'_2})}(\lf_1, K_{\lf_1}.n_1)$. 
On the other hand, we have seen in Remark \ref{rkdefind} that our hypothesis implies that $J_1'$ weakly induces $K_{\lf_2}.n_2$ in $\lf_2'$. Then it follows from Theorem \ref{genpart}, applied to the symmetric pair $(\lf_2', \kk_{\lf_2'})$, that $n_2\in\overline{J_1'}$. Hence $\overline{J_1}\supset \cc_{\pp}(\lf_2)+n_2$ and we get $J_2\subset\overline{J_1}$ by $K$-stability of $\overline{J_1}$.
\end{proof}

A key point in the proof of \eqref{wiii}$\Rightarrow$\eqref{wiv} is that whenever we have $x\in J_2\cap\overline{J_1}$ with Jordan decomposition $x=s+n$, we can manage to realize the degeneration through a one parameter family $(x_t)_{t\in \K^*}\in J_1$, $\lim_{t\rightarrow0} x_t= x$ with the $x_t$ lying in a Slodowy slice centered on $x$ and included in $\g^s$. Hence the whole degeneration takes place in the Levi $\g^s$ and thus yields a degeneration toward any element with Jordan decomposition similar to $x$.


\begin{lm}\label{dimind}
Let $J_1$, $J_2$ be $K$-decomposition classes such that $J_1\wupind J_2$. Then $J_1\upind J_2$ if and only if the dimension of $K$-orbits in $J_1$ and $J_2$ coincide.
\end{lm}
\begin{proof}
Choose a datum $(\lf_2,\Od_2)$ defining $J_2$.  We can argue as in Theorem \ref{thmind} (ii)$\Rightarrow$(i) to find a data $(\lf_1,\Od_1)$ defining $J_1$ and an element $y_1\in J_1\cap U_{\lf_2}\cap J_{(\lf_2,\kk_{\lf_2})}(\lf_1,\Od_1)$. 
On the other hand, any element $y_2\in \cc_{\pp}(\lf_2)^\bullet+n_2$, $n_2\in \Od_2$ satisfies  $y_2\in J_2\cap U_{\lf_2}$. 
Then, it follows from Lemma \ref{ts}(iii) that $y_1$ and $y_2$ share the same $K$-orbit codimension in $\pp$ if and only if they share the same $K_{\lf_2}$-orbit codimension in $\pp_{\lf_2}$. Since $\dim K_{\lf_2}.y_2=\dim \Od_2$, the result follows.
\end{proof}

As a consequence of Theorem \ref{thmind} and Lemma \ref{dimind}, we are now able to prove the following. 

\begin{cor}\label{corind}
\begin{enumerate}[(i)]

\item If $J_0$ is a $K$-decomposition class, then 
\[\overline{J_0}=\bigcup_{J\wupind J_0} J,\qquad \overline{J_0}^{\bullet}=\bigcup_{J\upind J_0} J.\] 
\item If $J_1$ and $J_2$ are $K$-decomposition classes, $\overline{J_1}\cap \overline{J_2}$ is a union of $K$-decomposition classes.
\item Sheets are union of $K$-decomposition classes.
\item Induction (resp. weak induction) of data and of $K$-decomposition classes are transitive. That is, for classes, 
\[(J_1\upind J_2)\wedge (J_2\upind J_3)\Rightarrow  J_1\upind J_3,\]
\[\textrm{resp. }(J_1\wupind J_2)\wedge (J_2\wupind J_3)\Rightarrow  J_1\wupind J_3.\phantom{re}\]
\item \label{equivind} In the Lie algebra case, the parabolic induction coincides with the slice induction.
\end{enumerate}
\end{cor}
\begin{proof}
Since $K$-decomposition classes form a partition of $\pp$, the first part of (i) is an immediate consequence of the equivalence \eqref{wi}$\Leftrightarrow$ \eqref{wiii}$\Leftrightarrow$ \eqref{wiv} in Theorem \ref{thmind}. The second part of (i) follows from the first one and Lemma \ref{dimind}. Then we deduce from (i) the statement (iv) concerning $K$-decomposition classes. 

Let us now show (iv) on data. Choose three data $(\lf_i,\Od_i)$, $i=1,2,3$ such that $(\lf_i,\Od_i)$ (weakly) induces $(\lf_{i+1}, \Od_{i+1})$, $i=1,2$. Without loss of generality, we can assume that $\g=\lf_3$ (see Remark \ref{rkdefind}). Then $J(\lf_i, \Od_i)$ (weakly) induces $J(\lf_{i+1}, \Od_{i+1})$ for $i=1,2$, so $J(\lf_1, \Od_1)$ (weakly) induces $J(\lf_3, \Od_3)$. Hence there exist data $(\widetilde{\lf_i}, \widetilde{\Od_i})$, $K$-conjugate to $(\lf_i, \Od_i)$ ($i\in \{1,3\})$, and such that $(\widetilde{\lf_1}, \widetilde{\Od_1})$ (weakly) induces $(\widetilde{\lf_3},\widetilde{\Od_3})$. Since $\lf_3=\g$, we have $(\widetilde{\lf_3}, \widetilde{\Od_3})=(\lf_3, \Od_3)$ so, up to $K$-conjugacy, we can assume that $(\widetilde{\lf_1}, \widetilde{\Od_1})=(\lf_1, \Od_1)$.

In order to show (ii), we only need to note that $\overline{J_1}\cap \overline{J_2}$ is the union of $K$-decomposition classes induced both by $J_1$ and $J_2$. Statement (iii) follows from the fact that any sheet is the regular closure of a $K$-decomposition class.

In \cite[3.5-3.6]{Bo} it is shown that, whenever $J_0$ is a $G$-decomposition class, $\overline{J_0}^{\bullet}$ is the set of $G$-decomposition classes parabolically induced by $J_0$. Then \eqref{equivind} for $G$-decomposition classes follows from (i) and the discussion in Remark \ref{rqdec}. We now prove \eqref{equivind} for data, using arguments similar to the proof of (iv) for data. Without loss of generality, we can restrict to the cases where a datum $(\lf_1,\Od_1)$ slice (resp. parabolically) induces a datum of the form $(\g,\Od_2)$. Since $(\g,\Od_2)$ is the only datum defining $J(\g,\Od_2)$, such a slice (resp. parabolic) induction is equivalent to the corresponding induction from $J(\lf_1,\Od_1)$ to $J(\g,\Od_2)$. The result on data then follows from the result on decomposition classes.
\end{proof}

An important consequence of Corollary \ref{corind} \eqref{equivind} is that the slice induction can be seen as a generalization of the parabolic induction, fitting to the symmetric Lie algebra setting needs. Note in particular that (i), (ii), (iv) and (iii) are respective analogues of \cite[3.5, 3.8, 2.3]{Bo} and \cite[5.8.d]{BK}.

The fact that a closure of a decomposition class is a union of $K$-decomposition classes has also been shown in \cite{Le} when the ground field is precisely $\C$.

\begin{ex}Using the results of Section \ref{slosli} and of this section, we can reconsider the example given in \cite[\S14(5)]{Bu2} (where there is a mistake in the definition of $\lf$). Here $(\g,\kk)=(\mathfrak{sp}_6,\gl_3)$, the sheet $S_G$ of $\g$ is defined by $S_G=\overline{J_{\g}(\lf,0)}^{\bullet}\subset\g^{(16)}$ with $\lf$ arising from $\pp$ of type $\widetilde A_1$ (long root). We know that $J_{\g}(\lf,0)\cap\pp$ is the single $K$-decomposition class $J_0:=J_{(\g,\kk)}(\lf,0)$ \cite[Theorem~7.8]{Bu2}. In particular, $S_G\cap\pp$ has only one irreducible component of maximal dimension $8+\dim \cc_{\pp}(\lf)=10$, namely $\overline{J_0}^{\bullet}$ \cite[Lemma~8.2]{Bu2}. Any nilpotent element $e_0\in \pp^{(8)}\subset\g^{(16)}$ embedded in a normal $\sld$-triple $\sS_0$ satisfies $e_0\in \overline{J_0}^{\bullet}$ if and only if $\dim X(\overline{J_0}^{\bullet}, \sS_0)=\dim J_0-8=2$ (see Theorem \ref{genpart} and Proposition \ref{propsmooth}). 
We choose the two nilpotent elements $e,e'\in \pp^{(8)}$ as in \cite[\S14(5)]{Bu2} and embed them in $\sld$-triples $\sS,\sS'$. It turns out that the Slodowy slices $X(S_G\cap\pp,\sS)$ and $X(S_G\cap\pp, \sS')$ are irreducible varieties of respective dimension $1$ and $2$. In particular $J_0\not\upind K.e$ while $J_0\upind K.e'$. An other consequence is that $S_G\cap\pp$ has at least two irreducible components, $\overline{J_0}^{\bullet}$ of dimension $10$ and $\overline{K.(e+X(S_G\cap\pp, \sS))}^{\bullet}$ of dimension $1+8=9$, see Proposition \ref{deformation1}. This is a counterexample to the equidimensionality of $S\cap\pp$ for $S$ a sheet of $\g$, this property being observed when $\g=\gl_n$ \cite[Theorem~13.2]{Bu2}
\end{ex}

\section{Sheets of $\g$ in type G$_2$}
\label{secG2}
In this section, we assume that $\g$ is a simple Lie algebra of type G$_2$.  We are interested in studying non-trivial (i.e. non-regular and not restricted to one nilpotent orbit) sheets in this case. We adopt conventions and notations of \cite{FH}. In particular, we fix a Cartan subalgebra $\h$ of $\g$. In the corresponding root system, we fix a basis $\{\alpha_1,\alpha_2\}$, with $\alpha_1$ a short root,  and label the associated positive roots as pictured

%
%

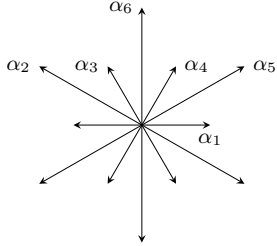
\tikzset{triplearrow/.style={
  thick,
  double distance=4pt, 
  -{stealth}}, 
thirdline/.style={
thick, -{stealth} }
}

\begin{center}
\scalebox{0.9}{\begin{tikzpicture}[>=stealth]
\draw[->] (0,0) -- (1,0) node[below]{\small{$\alpha_1$}};
\draw[->] (0,0) -- (-3/2,0.866) node[left]{\small{$\alpha_2$}};
\draw[->] (0,0) -- (-1/2,0.866) node[left]{\small{$\alpha_3$}};
\draw[->] (0,0) -- (1/2,0.866) node[right]{\small{$\alpha_4$}};
\draw[->] (0,0) -- (3/2,0.866) node[right]{\small{$\alpha_5$}};
\draw[->] (0,0) -- (0,1.732) node[left]{\small{$\alpha_6$}};
\draw[->] (0,0) -- (-1,0);
\draw[->] (0,0) -- (3/2,-0.866);
\draw[->] (0,0) -- (1/2,-0.866);
\draw[->] (0,0) -- (-1/2,-0.866);
\draw[->] (0,0) -- (-3/2,-0.866);
\draw[->] (0,0) -- (0,-1.732);



\end{tikzpicture}}

%
\end{center}

We choose, as a Chevalley basis,  the one given in \cite[p.346]{FH} and denote it by $h_1$, $h_2$, $x_i (i \in[\![1,6]\!])$, $y_i (i\in[\![1,6]\!])$, with $x_i\in \g_{\alpha_i}$, $y_i\in \g_{-\alpha_i}$ and $(x_i,h_i,y_i)$ an \Striplet for $i=1,2$.

The nilpotent cone of $\g$ consists of five nilpotent conjugacy classes. We denote the orbit of dimension $2i$ ($i\in\{0,3,4,5,6\}$) by  $\Omega_{2i}$.
Moreover, we choose some particular representatives $n_{i}$ (two of them when $i=5$) of these orbits as given in column 2 of Table~ \ref{representants}.

\begin{table}[ht]
\centering
$\begin{array}{|c| c| c| }
\hline
G\textrm{-orbit }\Omega& \textrm{Representative}& K\textrm{-orbit }\Od\\
\hline
\Omega_0& 0&\Od_0\\
\hline
\Omega_6 & n_3:=x_5 & \Od_3\\
\hline
\Omega_8& n_4:=x_4& \Od_4\\
\hline
\multirow{2}{*}{$\Omega_{10}$} & n_{5a}:=x_5+y_3& \Od_{5a}\\
\cline{2-3}
& n_{5b}:=x_2+x_5&\Od_{5b}\\
\hline
\Omega_{12} & n_{6}:=x_3+y_2& \Od_{6}\\
\hline
\end{array} $
\caption{Representatives for $\Omega_i$ and $\Od_i$}
\label{representants}
\end{table}


Since $\g$ is of rank two, the only non-regular non-nilpotent elements of $\g$ are semisimple subregular elements. This gives rise to two subregular sheets of $\g$ which can be described as union of $G$-decomposition classes as follows.
\[
 S^{\g}_1:=G.(\K^* h_1) \sqcup \Omega_{10}, \qquad  S^{\g}_2:=G. (\K^* h_2) \sqcup \Omega_{10}.
\]
This follows from the fact that $G.(\K^* h_1)$ and $G.(\K^* h_2)$ are of the same dimension (here, $11$) and that both must contain a nilpotent orbit in their regular closure (Proposition \ref{deformation0}). A practical criterion to distinguish generic elements of $S^{\g}_1$ from those of $S^{\g}_2$ is that elements of the former lie in a Levi of type $A_1$ (short root) while their centraliser is a Levi of type $\widetilde{A_1}$ (long root).

It is known \cite{Sl, Pe, BK, Hi} that $S^{\g}_1$ is smooth at points of $G.(\K^* h_1)$ and is singular at points of $\Omega_{10}$, undergoing a threefold covering in a resolution of singularities,  see \emph{e.g.} Proposition \ref{detg}. On the other hand, $S^{\g}_2$ is a smooth variety.

Letting \[\kk:=\h\oplus\bigoplus_{i\in \{1,6\}}\g_{\pm\alpha_i}, \qquad \pp:=\bigoplus_{i\in \{2,3,4,5\}}\g_{\pm \alpha_i},\]
we construct a symmetric Lie algebra of type G$_{2(2)}$. It corresponds to the single non-compact form for an algebra of type G$_2$ and we have $\kk\cong \sln_2\oplus\sln_2$.
 
Let us describe the $K$-orbits of interest in this setting. First, the symmetric Lie algebra is of maximal rank, hence any $G$-orbit intersects $\pp$ \cite{An}. Since semisimple orbits intersect $\pp$ into single orbits (see \emph{e.g.}  \cite[Proposition~6.6]{Bu2}), there are exactly two subregular semisimple $K$-decomposition classes: $J_i:=K.(\K^*\tilde h_i)$, with $\tilde h_i$ an element of $G.h_i\cap \pp$, $i=1,2$. It turns out \cite{Dj1} that $\Omega_{2i}\cap\pp$ is a single $K$-orbit $\Od_i$ for $i\in\{0,3,4,6\}$ and is the union of two $K$-orbits $\Od_{5a}$ and $\Od_{5b}$ in the subregular case $i=5$ (respectively numbered by $3$ and $4$ in \cite{Dj1}). 
We refer to Table \ref{representants} for representatives of these different $K$-orbits.

Note that, since sheets of $\pp$ are union of $K$-decomposition classes, there are two 6-dimensional sheets of subregular elements. These are $S_1:=\overline{J_1}^{\bullet}$ and $S_2:=\overline{J_2}^{\bullet}$. Even if each of these sheets must contain at least a nilpotent orbit among $\Od_{5a}$ and $\Od_{5b}$ (Proposition \ref{deformation0}), it is a non-trivial  problem to decide which orbits belong to a given sheet. At this point, we may even not rule out the existence of a possible $5$-dimensional subregular sheet consisting of a single orbit $\Od_{5a}$ or $\Od_{5b}$.

\subsection{Applications of the Slodowy slice theory}
\label{G2slsl}
We now apply partly the general theory of Sections \ref{slosli} and \ref{sliceind} to our special case. The main result of this subsection is the following
\begin{prop}\label{descsheetG2}
\begin{enumerate}[(i)]
\item  The only sheets of subregular elements in $\pp$ are $S_1$ and $S_2$.
\item The decomposition of each sheet $S_i$ ($i=1,2$)  as union of $K$-decomposition classes is \[S_i=J_i\sqcup \Od_{5a}\sqcup \Od_{5b}\quad(=S_i^{\g}\cap \pp).\]
\item The sheet $S_1$ is smooth on $J_1$ and $\Od_{5a}$. At points of $\Od_{5b}$, the singularities of $S_1$ are smoothly equivalent to the intersection of three lines.
\item The sheet $S_2$ is smooth. 
\end{enumerate}
\end{prop}

Recall that the only decomposition classes of subregular elements are $J_1$, $J_2$, $\Od_{5a}$ and $\Od_{5b}$. So, since sheets are regular closures of decomposition classes, (ii) implies (i).

The statements (ii), (iii) and (iv) rely on computations on the Slodowy slices. In fact, embedding $n_{5a}$ and $n_{5b}$ into respective normal $\sld$-triple, $\sS_{5a}$ and $\sS_{5b}$, we claim that 
\begin{lm} \label{eXG2}
\begin{enumerate}[(i)]
\item $X(\pp^{(5)},\sS_{5a})$ is the union of two lines $\K t_i$ ($i=1,2$), with $n_{5a}+t_1\in J_1$ and $n_{5a}+t_2\in J_2$.
\item $X(\pp^{(5)},\sS_{5b})$ is the union of four lines $\K t_i$ ($i=3,4,5,6$), with $n_{5b}+t_i\in J_1$ for $i\in \{3,4,5\}$ and $n_{5b}+t_6\in J_2$.
\end{enumerate}
\end{lm}

In particular, $X(J_i,\sS_{5x})\neq\emptyset$ so $J_i$ slice induces $\Od_{5x}$ for $i\in \{1,2\}$ and $x\in \{a,b\}$. Hence Proposition \ref{descsheetG2} (ii) follows from Corollary \ref{corind} (or, in a more simple way, from Lemma \ref{lmelimX}).
 
The other consequence of Lemma \ref{eXG2} is that $J_i\sqcup \Od_{5a}$ ($i=1,2$)  (resp. $J_2\sqcup \Od_{5b}$) is smoothly equivalent to the (smooth) affine line $X(S_i, \sS_{5a})$ (resp. $X(S_2, \sS_{5b})$), see Remark \ref{smoothequiv}. On the other hand $J_1\sqcup \Od_{5b}$ is smoothly equivalement to a union of three lines meeting at a point. This explains (iii) and (iv) of  Proposition \ref{descsheetG2}.

To sum up, the picture is the following. Only one branch in the neighborhood of the singularities of $S_1^{\g}$ at points of $\Od_{5a}$  is preserved under the intersection with $\pp$. On the contrary, the singularities at points of $\Od_{5b}$ are ``intact'' when intersected with $\pp$.

\medskip

The idea of the proof of Lemma~\ref{eXG2} is the following. Due to Lemma \ref{Ft}, an element in the Slodowy slice is subregular if and only if it is not regular. In particular, we get explicit equations of the subregular locus in the Slodowy slice via minors of some matrices. Then, these equations are simple enough to solve. Most of the computations have been made by hand and then checked by \cite{GAP} using W.~de~Graaf's package \cite{SLA}\footnote{These GAP computations are available at \url{https://hal.archives-ouvertes.fr/hal-01017691v2/file/gap_g2_v6.txt}.}. We give the explicit outcome of the computation in terms of our Chevalley basis.  Set $m_{5a}:=x_3+y_5$ and $m_{5b}:=-\frac 43 y_2+\frac 23 y_3+\frac 23 y_4+\frac 43 y_5$. Then $\sS_{5a}:=(n_{5a}, -2h_2, m_{5a})$ and $\sS_{5b}:=(n_{5b},2h_1+4h_2 ,m_{5b})$ are normal $\sld$-triples and 
the elements $t_i$ of the lemma can be chosen as follows.
\begin{table}[ht]
\centering
$\begin{array}{|c| c|}
\hline
t_1&x_3+9y_5\\
\hline
t_2 & x_3+y_5\\
\hline
\end{array}
\qquad \qquad
\begin{array}{|c|c|}
\hline
t_3 & y_3+y_4\\
\hline
t_4&6y_2-3y_3+y_4\\
\hline
t_5& y_3-3y_4-6y_5\\
\hline
t_6&2y_2-y_3-y_4-2y_5\\
\hline
\end{array} $\qquad
\label{ti}
\end{table}


\begin{Rqs}
\begin{enumerate}[-]
\item Note that, thanks to \eqref{Ft}, our method also provides explicit one-parameter degenerations $\lim_{\gamma\rightarrow 0}n_{5x}+\gamma t_i= n_{5x}.$ ($x\in \{a,b\}$)

\item In the real case, we have no reason to think that decomposition classes closure inclusions always leaves footprints on the Slodowy slice level (see proof of the crucial statement Proposition \ref{propsmooth}(iii)). However, it is an interesting feature to note that our method may still provide real degeneration in some cases. Indeed, here, all the $t_i$ belong to the obvious real form of $\g$.

\item The computations presented in this section can be improved in order to get similar results in higher rank. For example, in a work in progress, the first author has found two singular sheets in the (ordinary) Lie algebra of type F$_4$, the others being smooth. The singular sheets are those appearing in line 9 and 10 in the last table of \cite{Br1}.
\end{enumerate}
\end{Rqs}

%



\subsection{$4$-ality and projections}
\label{G24ality}
We now provide an explicit desingularization of $S_1$ through the well-known projection $\so_7\twoheadrightarrow \g$. We first note that it would be hard to follow Broer's approach \cite{Br1,Br2} in the Lie algebra case for normalization and desingularization since it uses fiber bundles of the form $G\times^PY$ with $P$ a parabolic subgroup of $G$.

In order to achieve our goal, we use an other description of our algebra $\g$ of type G$_2$ using $4$-ality as defined in \cite[\S3.4]{LM}. Namely, we choose four copies $C_1,C_2,C_3, D$ of a $2$-dimensional space equipped with a non-degenerate bilinear skew-symmetric form $\omega$. Then
\[\g_0:=\sln(C_1)\times \sln(C_2)\times\sln(C_3)\times\sln(D)\oplus(C_1\otimes C_2\otimes C_3\otimes D) 
\]
can be equipped with a Lie bracket in such a way that $\g_0\cong \so_8$, see \eqref{so8} for an identification. In this model, $\kk_0:=\sln(C_1)\times \sln(C_2)\times\sln(C_3)\times\sln(D)$ is a Lie subalgebra of $\g_0$ and acts on $\pp_0:=C_1\otimes C_2\otimes C_3 \otimes D$ in the usual way. We refer to \cite{LM} for the definition of the bracket of two elements of $\pp_0$ using $\omega$ which, for instance, identifies $C_1$ with its dual and hence $\sld(C_1)$ with $S^2C_1$.

This presentation of $\so_8$ relies on the $\mathop{\mathfrak{S}}\nolimits_4$-symmetry of $\so_8$ as can be seen on the extended Dynkin diagram of type $D_4$. We are interested in the $\mathop{\mathfrak{S}}\nolimits_2$ (resp. $\mathop{\mathfrak{S}}\nolimits_3$)-action on $\g_0$ induced by permutations on the $2$ spaces $C_1$ and $C_2$ (resp. on the $3$ spaces $C_i$). Its fixed point space $\g_1\subset\g_0$ (resp. $\g_2\subset\g_1$) is isomorphic to $\so_7$ (resp. $\g$) and can be described as 
\[\g_i=\kk_i\oplus \pp_i \textrm{ with } \left\{\begin{array}{l l} \kk_1:=\sln(C')\times\sln(C_3)\times \sln(D)& \pp_1:=S^2 C'\otimes C_3\otimes D \\ \kk_2:=\sln(C)\times \sln(D)&  \pp_2:=S^3C\otimes D\end{array}\right.,\] where $C$ and $C'$ are again copies of our two-dimensional space. 

Choosing a Cartan subalgebra of $\g_2$ in $\sln(C)\times\sln(D)$ and weight vectors $c_{-},c_{+}$ (resp. $d_{-},d_+$) in $C$ (resp. $D$), we rediscover the combinatorics of the root spaces in $\g_2$ as pictured below.
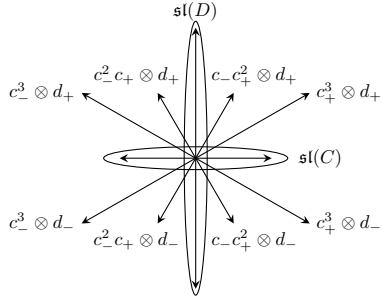
\begin{center}
\begin{tikzpicture}[>=stealth]
\tikzstyle{every node}=[scale=0.7]
\draw[->] (0,0) -- (1,0) node[right]{$\;\;\;\;\sln(C)$};
\draw[->] (0,0) -- (-3/2,0.866) node[left]{{$c_{-}^3\otimes d_+$}};
\draw[->] (0,0) -- (-1/2,0.866) node[above]{{$c_{-}^2c_+\otimes d_+$}\;\;\;\;\;\;\;\;};
\draw[->] (0,0) -- (1/2,0.866) node[above]{\;\;\;\;\;\;\;\;{$c_{-}c_+^2\otimes d_+$}};
\draw[->] (0,0) -- (3/2,0.866) node[right]{{$c_{+}^3\otimes d_+$}};
\draw[->] (0,0) -- (0,1.732) node[above]{$\sln(D)$};
\draw[->] (0,0) -- (-1,0);
\draw[->] (0,0) -- (3/2,-0.866) node[right]{{$c_{+}^3\otimes d_{-}$}};
\draw[->] (0,0) -- (1/2,-0.866) node[below]{\;\;\;\;\;\;\;\;{$c_{-}c_+^2\otimes d_{-}$}};
\draw[->] (0,0) -- (-1/2,-0.866) node[below]{{$c_{-}^2c_+\otimes d_{-}$}\;\;\;\;\;\;\;\;};
\draw[->] (0,0) -- (-3/2,-0.866) node[left]{{$c_{-}^3\otimes d_{-}$}};
\draw[->] (0,0) -- (0,-1.732); 
\draw (0,0) ellipse (8ex and 1ex);
\draw (0,0) ellipse (1ex and 12ex);
\end{tikzpicture}\end{center}


Moreover, we have a natural projection $\pi:\g_1\twoheadrightarrow \g_2$, given for instance on $\pp_1$ via $\pi(xy\otimes z\otimes t) =xyz\otimes t$.
%
Its kernel is a $\kk_2=\sln(C)\times \sln(D)$-module so 
$\pi$ is $SL(C)\times SL(D)$-equivariant.

We also introduce faithful representations that allow us to manipulate more easily these algebras.
Let $V_8:=C_1\otimes C_2\oplus C_3\otimes D$ with action of $\pp_0$ given by
\[(c_1\otimes c_2\otimes c_3\otimes d).(c_1'\otimes c_2'+c_3'\otimes d)=\omega (c_1,c_1')\omega(c_2,c_2')c_3\otimes d+\omega (c_3,c_3')\omega(d,d')c_1\otimes c_2\]
and standard action of $\kk_0$. It is one of the three inequivalent fundamental representations of $\g_0\cong \so_8$ of dimension $8$ and $\g_0\subseteq\gl(V_8)$ is the subalgebra preserving the symmetric form \begin{equation*}\label{so8}\omega(c_1\otimes c_2+c_3\otimes d, c_1'\otimes c_2'+c_3'\otimes d')=\omega(c_1,c_1')\omega(c_2,c_2')-\omega(c_3,c_3')\omega(d,d').\end{equation*}
As an $\sln(C)\times \sln(D)$-module, $V_8$ decomposes as $V_1\stackrel{\perp}{\oplus} V_a\stackrel{\perp}{\oplus} V_b$ where $V_1=\Lambda^2C'$, $V_a=S^2C'$ and $V_b=C_3\otimes D$. Setting $V_7:=V_a\oplus V_b$, one gets that $\g_1$ is the subalgebra of $\so_8$ stabilizing both subspaces $V_1$ and $V_7$ while $\kk_1$ is the subalgebra of $\g_1$ stabilizing both subspaces $V_a$ and $V_b$. In particular, we recover that $\g_1\cong \so_7\oplus\so_1\cong\so_7$ and that $(\g_1,\kk_1)$ is a symmetric pair isomorphic to $(\so_7,\so_4\oplus\so_3)$ \cite{PT}. 


We denote by $\Omega'_{10}$ the $SO_7$-orbit of dimension $10$ in $\g_1$. Its Young diagram is {\tiny \yng(3,1,1,1,1)}. 
Moreover, we let $\mathcal T$ be the set  elements of $\g_1$ of rank at most $2$ with respect to the representation on $V_7$. The closed set $\mathcal T$ is nothing but the union $SO_7.(\K^*t)\sqcup\overline{ \Omega_{10}'}$ where $t$ is any semisimple element whose centraliser is a Levi of type B$_2$. In particular $\mathcal T^{\bullet}$ is a sheet in $\g_1$.

One can then extract the following information from \cite{LSm, Kr, Hi}.

\begin{prop}\label{detg}\begin{enumerate}[(i)]
\item $\pi$ induces finite surjective maps $\overline{\Omega'_{10}}\rightarrow \overline{\Omega_{10}}$, $\mathcal T\rightarrow \overline{S^{\g}_1}$ and a desingularization $\mathcal T\setminus \{0\}\rightarrow \overline{S^{\g}_1}\setminus\{0\}$.  
\item The cardinality of fibers is given by the following table:
\[\begin{array}{|c| c| c| c| c|c| }
\hline
y\in & G_2.(\K^*h_1) & \Omega_{10} & \Omega_8& \Omega_6 & 0\\
\hline
\# \;\pi^{-1}(y)\cap \mathcal T=&1 & 3& 2&1&1\\
\hline
\end{array} \]
\end{enumerate}
\end{prop}
%
%

We refer to \cite{Oh} for the classification of nilpotent orbits in $\pp_1$. 
In $\g_1$, the $SO_7$-orbit $\Omega_{10}'$ splits into two $SO_3\times SO_4$-orbits $\Od_{5a}'$ and $ \Od_{5b}'$ whose respective $ab$-diagrams are {\scriptsize \young(aba,a,b,b,b)} and {\scriptsize \young(bab,a,a,b,b)}. We also define $\mathcal T_{\pp_1}:=\mathcal T\cap \pp_1$.
We can then state the following result
\begin{prop}\label{desing}
\begin{enumerate}[(i)]
\item $\pi$ induces finite surjective maps $\overline{\Od'_{5x}}\rightarrow \overline{\Od_{5x}}$ ($x\in \{a,b\}$),  $\mathcal T_{\pp_1}\rightarrow \overline{S_1}$ and a desingularization $\mathcal T_{\pp_1}\setminus\{0\}\rightarrow \overline{S_1}\setminus \{0\}$
\item The cardinality of fibers is given by:
\[\begin{array}{|c| c| c| c| c|c| c|}
\hline
y\in & J_1 & \Od_{5a} &\Od_{5b}& \Od_4& \Od_3 & 0\\
\hline
\# \pi^{-1}(y)\cap \mathcal T_{\pp_1}=&1 &1& 3 & 2&1&1\\
\hline
\end{array} \]
\end{enumerate}
\end{prop}
\begin{proof}
(ii) An element of $\pp_1$ acts on $V_8$ in the following way:
\[\left(\begin{array}{@{}c@{\;} |c | c}\scriptstyle{0}& \scriptstyle{0}&\scriptstyle{0}\\[-3pt] \hline
\scriptstyle{0}& 0 & p\\ \hline \scriptstyle{0}& p^{\vee} &\phantom{a} 0 \rule{0pt}{15pt}\phantom{a}\end{array}\right)
\begin{array}{l}\scriptscriptstyle{V_1}\\V_a\\V_b\rule{0pt}{15pt} \end{array}\]
where $p\in \Hom(V_b,V_a)$ and $p^{\vee}$ is the dual of $p$ with respect to $\omega_{|V_a\oplus V_b}$. We are interested into rank $2$ elements of this form and these are exactly elements for which $p$ is of rank $1$. We make use of the $4$-ality setting to describe such elements. An element of $\pp_1=(S^2C')\otimes (C_3 \otimes D)$ induces a rank one element of $Hom(S^2C',C_3\otimes D)$ (and, equivalently, of $Hom(C_3\otimes D, S^2C')$) if and only if it is a pure tensor of the form $(\alpha c_{+}^2+\beta c_+c_-+\gamma c_-^2)\otimes (c_+\otimes d_1 +c_-\otimes d_2)$
with $\alpha,\beta,\gamma\in \K$ and $d_1,d_2\in D$.  

After computation, one gets Table \ref{fibers}
where the column $I$ presents the $K$-decomposition class of $\overline{S_1}$ under consideration. 
Up to easy $SL(C)\times SL(D)$-conjugation, one can turn our representatives of Table \ref{representants} to the element $y\in \pp_2$ in column $II$. Finally, column $III$ present all elements in the fiber $\pi^{-1}(y)\cap \mathcal T_{\pp_1}$.

\begin{table}[ht]
\centering
$\begin{array}{|c| c| c| c| c|c| c| }
\hline
I& II& III\\
\hline
J_1=K.(\K^*(x_4+y_4)) &c_+^2c_-\otimes d_++c_+c_-^2\otimes d_-&c_+c_-\otimes(c_+\otimes d_++c_{-}\otimes d_-)\\
\hline
\Od_{5a} & c_+^3\otimes d_++c_+^2c_-\otimes d_-&c_+^2\otimes(c_+\otimes d_++c_-\otimes d_-)\\
\hline
\multirow{3}{*}{$\Od_{5b}$}&\multirow{3}{*}{$c_+^2c_-\otimes d_++c_+c_{-}^2\otimes d_+$}& c_+c_-\otimes(c_++c_-)\otimes d_+\\
\cline{4-4}
&& c_+(c_++c_-)\otimes c_-\otimes d_+\\
\cline{4-4}
&& (c_++c_-)c_-\otimes c_+\otimes d_+\\
\hline
\multirow{3}[-3]{*}{$\Od_4$}  & \multirow{3}[-3]{*}{$c_+^2c_-\otimes d_+$}& c_{+}^2\otimes c_-\otimes d_+\\
\cline{4-4}
&&c_+c_-\otimes c_+\otimes d_+\\
\hline
\Od_3 &c_+^3\otimes d_+&c_+^2\otimes c_+\otimes d_+\\
\hline
0& 0&0\\

\hline
\end{array} $
\caption{fibers of the form $\pi^{-1}(y)\cap \mathcal T_{\pp_1}$}
\label{fibers}

\end{table}

%
(i) The desingularization property follows from $\# \pi^{-1}(y)=1$ for $y\in J_1$ and the fact that $\mathcal T_{\pp_1}\setminus \{0\}$ is a smooth variety. Let us check this last point. Under the natural isomorphism $\g_1\cong \Lambda^2 V_7$, we have $\pp_1=V_a\wedge V_b$ and $\mathcal T_{\pp_1}=\{v_a\wedge v_b| v_a\in V_a,v_b\in V_b\}$. The smoothness is then a consequence of the transitivity of the action of $GL(V_a)\times GL(V_b)$ on $\mathcal T_{\pp_1}\setminus \{0\}$.

In view of Proposition \ref{detg}, there only remains to see that $\pi(\overline{\Od'_{5x}})=\overline{\Od_{5x}}$ for $x\in \{a,b\}$. 
For instance, since $c_+,c_-,d_+,d_-$ are isotropic vectors with respect to $\omega$, we can check that the pre-image $z:=c_+^2\otimes (c_+\otimes d_++c_-\otimes d_-)$ of $n_{5a}$ sends $c_-^2\in V_a$ on a non-zero multiple of $c_+\otimes d_++c_-\otimes d_-\in V_b$ and this last one is sent on a non-zero multiple of $c_+^2\in V_a$. Hence the $ab$-diagram of $z$ is {\scriptsize \young(aba,a,b,b,b)} and $z\in \Od_{5a}'$.

\end{proof}

\end{document}